\documentclass[11pt,a4paper,fleqn]{amsart}
\usepackage{amssymb}
\usepackage{hyperref}
\usepackage{natbib}

\newcommand{\ee}{\mathbb{E}}

\newcommand{\pp}{\mathbb{P}}

\DeclareMathOperator{\sgn}{sgn}
\DeclareMathOperator{\Cov}{Cov}
\newcommand{\cn}{\mathcal{N}}
\newcommand{\cl}{\mathcal{L}}

\newcommand{\dto}{\stackrel{d}{\to}}
\newcommand{\de}{\stackrel{d}{=}}
\newcommand{\toi}{\to\infty}

\newcommand{\asto}{\stackrel{a.s.}{\to}}

\newtheorem{theorem}{Theorem}
\newtheorem{lemma}[theorem]{Lemma}
\theoremstyle{remark}
\newtheorem{remark}{Remark }

\begin{document}
\bibliographystyle{plainnat}
\setcitestyle{numbers}
\title{On the functional limits for partial sums under stable law}
\author{Khurelbaatar Gonchigdanzan}
\address{Department of Mathematical Sciences, University of Wisconsin, Stevens Point, Wisconsin 54481, USA}
\email{hurlee@uwsp.edu}

\author{Kamil Marcin Kosi\'nski}
\address{Korteweg-de Vries Institute for Mathematics, Universiteit van Amsterdam, Plantage Muidergracht 24, 1018 TV
Amsterdam, The Netherlands}
\curraddr{Korteweg-de Vries Institute for Mathematics, University of Amsterdam, P.O. Box 94248,
1090 GE Amsterdam, The Netherlands and\\
EURANDOM, Technishe Universiteit Eindhoven, Postbus 513, 5600 MB Eindhoven, The Netherlands}
\email{K.M.Kosinski@uva.nl}

\date{May 7, 2009}
\subjclass[2000]{Primary G0F05}

\begin{abstract}
For the partial sums $(S_n)$ of independent random variables we define a stochastic process
$s_n(t):=(1/d_n)\sum_{k \le [nt]} ({S_k}/{k}-\mu)$ and prove that
$$(1/{\log N})\sum_{n\le N}(1/n)\mathbf {I}\bigl\{s_n(t)\le x\bigr\} \to G_t(x)\quad \text{a.s.}$$ if and only if
$(1/{\log N})\sum_{n\le N} (1/n)\pp\bigl(s_n(t)\le x\bigr) \to G_t(x)$,
for some sequence $(d_n)$ and distribution $G_t$.
We also prove an almost sure functional limit theorem for the product of partial sums of i.i.d. positive random variables attracted to an $\alpha$-stable law with $\alpha\in (1,2]$.
\end{abstract}

\maketitle

\section{Introduction and main result}
In the past two decades many interesting extensions of the classical central limit theorem (CLT) have been obtained. One of the extensions is known as almost sure central limit theorem (ASCLT) which is discovered by Brosamler (1988) and Schatte (1988) and has been extensively studied for independent random variables as well as dependent variables.  Motivated by ASCLT, almost sure versions of many limit theorems in probability and statistics have been obtained in the past. It is known that for i.i.d. r.v.'s ASCLT holds under the same assumptions as CLT but in general, the existence of the weak limit does not always imply the almost sure limiting result. For more discussions about the  early results on ASCLT we refer to \citet{Berkes1}.

In this note we consider the product of partial sums, denoted by $S_n$, of a sequence of random variables attracted to a stable distribution and its limit distributions. 
\citet{Rempala} established the limit distribution of the product of partial sums of a sequence of i.i.d. positive r.v.'s with mean $\mu$ and variance $\sigma^2$:
\begin{equation}
\label{eq1}
\left(\dfrac{\prod_{k=1}^{n}S_k}{n!\mu^n}\right)^{\mu/(\sigma
\sqrt n)}  \dto e^{\sqrt 2 \cn(0,1)}.
\end{equation}

\citet{Zhang} proves a weak invariance principle of \eqref{eq1} for i.i.d. r.v.'s. Recently, \citet{Kamil} has shown that the weak invariance principle still holds when the partial sums are attracted to an $\alpha$-stable law with $\alpha \in (1,2]$ which also generalizes the earlier 
result by \citet{Qi}.

Throughout this paper, $\log\log x$ and $\log x$ stand for $\ln\ln(\max\{x,e^e\})$ and
$\ln(\max\{x,e\})$ respectively. We also use the notations $a_n \ll b_n$ for $a_n=O(b_n)$ and $\mathbf I(A)$ for the indicator function on a set $A$.

Our main result in this note is to establish an almost sure version of the result by \citet{Kamil} that can generalize the early results by \citet{Hurlee1} and \citet{Hurlee2,Hurlee3}.

Recall that a sequence of and i.i.d. r.v.'s $\{X_n:n\ge1\}$ is said to be
in the domain of attraction of a stable law $\cl$ if there exist sequences $(a_n)$
and $(b_n)$ such that
\[
\frac{S_n-b_n}{a_n}\dto\cl_{\alpha},
\]
where $\cl_{\alpha}$ is one of the stable distributions with index $\alpha\in(0,2]$. Moreover,
let $\{\cl_{\alpha}(s):s\ge 0\}$ be the $\alpha$-stable L\'evy process corresponding to $\cl_\alpha$,
that is $\cl_{\alpha}(1)\de\cl_{\alpha}$.

The following theorem is well known (see, e.g., \citet{Hall}).
\begin{theorem}[Stability Theorem]
\label{stabthe}
The general stable law is given, to within type, by a characteristic function of one of the following forms:
\begin{enumerate}
\item $\phi(t)=\exp(-t^2/2)$ (normal case, $\alpha=2$);
\item $\phi(t)=\exp(-|t|^\alpha(1-i\beta\sgn(t)\tan(\frac{1}{2}\pi\alpha)))$
$(0<\alpha<1$ or $1<\alpha<2$, $-1\le\beta\le 1)$;
\item $\phi(t)=\exp(-|t|(1+i\beta\sgn(t)2/\pi\log|t|)$ $(\alpha=1$, $-1\le\beta\le 1)$.
\end{enumerate}
\end{theorem}
It is worth mentioning that in Theorem \ref{stabthe}, $\beta$ is the skewness parameter. In our paper, $\beta=1$ since $X_1$ is a positive random variable.

The first result of this note is the following almost sure functional limit theorem:
\begin{theorem}
\label{theorem1}
Let $\{X_n:n \ge 1\}$ be a sequence of i.i.d. positive random variables with $\ee X_1=\mu$ in the domain of attraction of an
$\alpha$-stable law $\cl_{\alpha}$ with $\alpha \in (1,2]$ and characteristic
function as in Theorem \ref{stabthe}. Define a process $\{\pi_n(t): 0 \le t \le 1\}$ by
\[
\pi_n(t):=\left(\prod_{k=1}^{[nt]} \dfrac{S_k}{\mu k}\right)^{{\mu}/{a_n}},
\]
where $(a_n)$ is a sequence of positive numbers that satisfies $({S_n-\mu n})/{a_n} \dto \cl_{\alpha}$ as $n\toi$.
Then for any real $x$
\[
\dfrac{1}{\log N}
\sum_{n=1}^{N}\dfrac{1}{n}{\mathbf I}\left(
\pi_n(t) \le x \right)\asto F_t(x)  \enskip \text{as} \enskip N \toi
\]
where $F_t$ is the distribution function of the random variable $\exp\left(\int_{0}^{t}\frac{\cl_{\alpha}(s)}{s}ds\right)$.
\end{theorem}
\begin{remark}
If $X_1$ has finite variance equal to $\sigma^2$ then $\alpha=2$, $\cl_\alpha\de\cn(0,1)$ and
$a_n\sim\sigma\sqrt n$, thus Theorem \ref{theorem1} implies the main
result of \citet{Hurlee3} which in particular yields the result
of \citet{Hurlee1} Theorem 2 since it is easy to verify that
\[
\int_{0}^{1}\frac{\cl_{\alpha}(s)}{s}ds\de\sqrt 2\,\cn(0,1).
\]
Moreover, \citet{Kamil} showed that for any $\alpha\in(1,2]$
\[
\int_{0}^{1}\frac{\cl_{\alpha}(s)}{s}ds\de(\Gamma(\alpha+1))^{1/\alpha}\cl_\alpha,
\]
hence Theorem \ref{theorem1} also yields the result of \citet{Hurlee2} Theorem 1.1.
\end{remark}

Our next result is the following Berkes-Dehling type of theorem (\citet[Theorem 2]{Berkes2}).
\begin{theorem}
\label{theorem2}
Let $\{Y_n:n\ge 1\}$ be a sequence of independent random variables and
$S_n=Y_1+\cdots+Y_n$. Let $(d_n)$ be a sequence of positive numbers such that
\begin{equation}
\label{eq3}
\frac{d_l}{d_k}\gg \left(\frac lk\right)^{\gamma}
\quad (l\ge k\ge n_0)
\end{equation}
for some $\gamma>0$ and $n_0 \ge 1$ and
\begin{equation}
\label{eq4}
\ee\left|\dfrac{S_n-\mu n}{d_n}\right| \ll e^{\gamma'(\log n)^{1-\varepsilon}}
\end{equation}
for some constant $\mu$ and $\gamma'\in(0,\gamma)$. Then for any distribution $G_t$,
\begin{equation}
\label{eq5}
\frac 1{\log N}\sum_{n=1}^N \frac1n
\mathbf {I}\biggl(\dfrac{1}{d_n}
\sum_{k=1}^{[nt]}\left(\dfrac{S_k}{k}-\mu\right)\le x\biggr)\asto G_t(x) \quad \text{as} \quad N \toi
\end{equation}
if and only if
\begin{equation}
\label{eq6}
\frac 1{\log N}\sum_{n=1}^N \frac1n
\pp\biggl(\dfrac{1}{d_n}
\sum_{k=1}^{[nt]}\left(\dfrac{S_k}{k}-\mu\right) \le x\biggr) \to G_t(x)\quad \text{as} \quad N \toi.
\end{equation}
\end{theorem}

\section{Auxiliary results}
The following three lemmas are needed for the proof of our main result.
\begin{lemma}[Lemma 2.3, \citet{Hurlee2}]
\label{lemma1}
Under the assumption of Theorem \ref{theorem1} we have
\[
\biggl|\dfrac{\mu}{a_n}\sum_{k=1}^{[nt]} \log\left(\dfrac{S_{k}}{\mu k}\right) - \dfrac{1}{a_n}\sum_{k=1}^{[nt]} \left(\dfrac{S_{k}}{k}-\mu\right)
\biggr| \asto 0 \quad \text{as} \quad n \toi.
\]
\end{lemma}
\begin{lemma}
\label{lemma2}
Under the assumptions of Theorem \ref{theorem1} we have
\[
\dfrac{1}{a_n}\sum_{k=1}^{[nt]}\left(\dfrac{S_k}{k}-\mu\right) \dto  \int_{0}^{t}\frac{\cl_{\alpha}(s)}{s}ds\quad\text{in}\quad D[0,1].
\]
\end{lemma}

\begin{proof}
This is a particular case of Theorem 2 in \citet{Kamil} when $f(x)=x$.
\end{proof}

\begin{lemma}
\label{lemma3}
Under the assumptions of Theorem \ref{theorem2} we have
\[
\ee\biggl(\dfrac{1}{d_n}\max_{1\le k \le n}\biggl|\sum_{j=1}^{k}\log\left(\dfrac{n+1}{j}\right)(Y_j-\mu)\biggr|\biggr)\ll \log n\,\ee \biggl(\dfrac{1}{d_n}\max_{1\le k \le n} |S_{k}-k\mu|\biggr)
\ll\log n\,e^{\gamma'(\log n)^{1-\varepsilon}}.
\]
\end{lemma}

\begin{proof}
The first part is Lemma 1 in \citet{Hurlee3} valid for any sequence of random variables.
The second part is Lemma 1 in \citet{Berkes2} combined with the assumption \eqref{eq4}.
\end{proof}

\section{Proofs of the main results}
To prove Theorem \ref{theorem1} we need the result in Theorem \ref{theorem2}. Let us prove Theorem \ref{theorem2} first, then Theorem \ref{theorem1} for convenience.
\begin{proof}[Proof of Theorem \ref{theorem2}]
\medskip
According to \citet{Berkes2} (p.\,1647) it suffices to prove that for any bounded Lipschitz function $g$ on $D[0,1]$ we have
\begin{equation}
\label{eq7}
\frac{1}{\log n}\sum_{k=1}^{n} \frac{1}{k}\left(g\left(\dfrac{s_k}{d_k}\right)-\ee g\left(\dfrac{s_k}{d_k}\right)\right)
\asto 0 \enskip \text{as} \enskip n \toi,
\end{equation}
where $s_n:=s_n(t)=\sum_{k\le [nt]}\left({S_k}/{k}-\mu\right)$.

It turns out that the following estimate is indeed sufficient for \eqref{eq7} (see p.\,1648 \cite{Berkes2} for the proof):
\begin{equation}
\label{eq8}
\ee\biggl(\sum_{k=1}^{n}\frac 1k\xi_k\biggr)^2\ll \log^2 n (\log \log n) ^{-1-\varepsilon} \enskip \text{for some}\enskip  \varepsilon>0,
\end{equation}
where $\xi_k=g(s_k/d_k)-\ee g(s_k/d_k)$.
\par Observe that
$\sum_{k=1}^{n}\left({S_k}/{k}-\mu\right)=\sum_{k=1}^{n}b_{k,n}(Y_{k}-\mu)$
where $b_{k,n}=\sum_{j=k}^{n}1/j$. It can be easily seen that
\[
s_l- s_k=b_{[kt]+1,[lt]}(S_{[kt]}-[kt]\mu)+\left(b_{[kt]+1,[lt]}(Y_{[kt]+1}-\mu)+\cdots + b_{[lt],[lt]}(Y_{[lt]}-\mu)\right)
\]
for $l\ge k$.

Obviously
$s_l- s_k-b_{[kt]+1,[lt]}(S_{[kt]}-\mu[kt])$ is independent of $s_k$, so we get
\[
\Cov\biggl(g\biggl(\frac{s_k}{d_k}\biggr),
g\biggl(\frac{s_l-s_k-b_{[kt]+1,[lt]}(S_{[kt]}-\mu[kt])}{d_l}\biggr)\biggr)=0 \enskip \text{for} \enskip l \ge k.
\]
Since $g$ is a bounded Lipschitz it follows that
\begin{align*}
\bigl| \ee(\xi_{k} \xi_{l})\bigr|&
= \biggl|\Cov\biggl(g\biggl
(\dfrac{s_{k}}{d_k}\biggr),g\biggl(\dfrac{s_{l}}{d_l}\biggr)-
g\biggl(\dfrac{s_l-s_{k}-b_{[kt]+1,[lt]}(S_{[kt]}-\mu[kt])}{d_l}\biggr)\biggr|
\\
& \ll \ee\left(\max_{0\le t \le 1}\dfrac{|s_{k}+b_{[kt]+1,[lt]}(S_{[kt]}-\mu[kt])|}{d_l}\right) \\
& \le  \ee\left(\max_{0\le t \le 1}\dfrac{|s_{k}|}{d_l}\right)+\ee\left(\max_{0\le t \le 1}\dfrac{|b_{[kt]+1,[lt]}(S_{[kt]}-\mu[kt])|}{d_l}\right) \\
& = \dfrac{d_k}{d_l}\left(\ee\left(\max_{0\le t \le 1}\dfrac{|s_k|}{d_k}\right)+\ee\left(\max_{0\le t \le 1}{b_{[kt]+1,[lt]}}\frac{| S_{[kt]}-\mu[kt]|}{d_k}\right)\right).
\end{align*}
Moreover, noticing $\max_{0 \le t \le 1}{b_{[kt]+1,[lt]}}=\log(l/k)$ and applying Lemma \ref{lemma3} we get
\begin{align*}
\bigl| \ee(\xi_{k} \xi_{l})\bigr|& \ll \dfrac{d_k}{d_l}\left(\ee\biggl(\max_{0\le t \le
 1}\frac{1}{d_k}\biggl|\sum_{i=1}^{[kt]}{b_{i,k}}(Y_{i}-\mu)\biggr|\biggr)+\log(l/k)\ee\left(\max_{0\le t \le 1}\frac{| S_{[kt]}-\mu[kt]|}{d_k}\right)\right) \\
& =
\dfrac{d_k}{d_l}\left(\ee\biggl(\max_{0\le j \le k}\frac{1}{d_k}\biggl|\sum_{i=1}^{j}{b_{i,k}}(Y_{i}-\mu)\biggr|\biggr)+\log(l/k)
\ee\left(\max_{0\le j \le k}\frac{| S_{j}-\mu j|}{d_k}\right)\right) \\
&  \ll\log l\,\dfrac{d_k}{d_l}\,\ee\left(\max_{1\le j \le k}\frac{| S_{j}-\mu j|}{d_k}\right)\ll
\log l\,\left(\dfrac{k}{l}\right)^{\gamma} \log k e^{\gamma'(\log k)^{1-\varepsilon}}=:c_{k,l}.
\end{align*}

On the other hand we also have $\ee(\xi_{k} \xi_{l}) \ll 1$ because $\xi_k$ is bounded. Hence we estimate $\ee(\xi_{k} \xi_{l})$ as follows:
\[
\ee(\xi_{k} \xi_{l}) \ll \left \{
			\begin{array}{rl}
      1,   & \text{if }  {l}/{k} \le \exp\bigl((\log n)^{1-\varepsilon}\bigr) \\
      c_{k,l},   & \text{if }  {l}/{k} \ge \exp\bigl((\log n)^{1-\varepsilon}\bigr)
      \end{array}
      \right.
\]
where $\varepsilon$ is any positive number.

Thus we get
\begin{equation}
\label{eq9}
\sum_{\substack{ 1\le k \le l\le n \\ {l}/{k} \le \exp{((\log n)^{1-\varepsilon}})}}\dfrac{\ee(\xi_k\xi_l)}{kl}\le \sum_{1 \le k \le n}\dfrac{1}{k}\sum_{k \le l \le ke^{(\log n)^{1-\varepsilon}}}\dfrac{1}{l} \ll \sum_{k=1}^{n}\dfrac{1}{k}\log^{1-\varepsilon}n \ll \log^{2-\varepsilon}n
\end{equation}
and
\begin{align}
\label{eq10}
\sum_{\substack{ 1\le k \le l\le n \\ {l}/{k} \ge \exp{((\log n)^{1-\varepsilon})} }}\dfrac{\ee(\xi_k\xi_l)}{kl}
&\le 
\log^2 n e^{\gamma'(\log n)^{1-\varepsilon}}
\sum_{\substack{ 1\le k \le l\le n \\ {l}/{k} \ge \exp{((\log n)^{1-\varepsilon})}}}\dfrac{1}{kl}\left(\frac{k}{l}\right)^\gamma\nonumber\\
&\le 
\log^2 n e^{(\gamma'-\gamma)(\log n)^{1-\varepsilon}}\sum_{1\le k\le l\le n}\dfrac{1}{kl}\nonumber\\
&\ll \log^4 n e^{(\gamma'-\gamma)(\log n)^{1-\varepsilon}}
\ll \log^{2-\varepsilon} n,
\end{align}
where the last estimation follows because $\gamma'\in(0,\gamma)$.

Since
\[
\ee\biggl(\sum_{k=1}^{n}\frac 1k\xi_k\biggr)^2\ll \sum_{1\le k \le l \le n} \dfrac{1}{kl}|\ee(\xi_k\xi_l)|
\]
by \eqref{eq9} and \eqref{eq10} it follows \eqref{eq8}.
\end{proof}
Before proving Theorem \ref{theorem1}, recall that it is well known that the sequence $(a_n)$ in Theorem \ref{theorem1} can be written as $a_n=n^{1/\alpha}L(n)$ where $L$ is a slowly varying function.
\begin{proof}[Proof of Theorem \ref{theorem1}]
We first show the equivalence of \eqref{eq5} and \eqref{eq6} under the conditions of Theorem \ref{theorem1} setting $d_n:=a_n$. In fact \eqref{eq4} is a direct consequence of Theorem 6.2 in \citet{DeAcosta}. \eqref{eq3} can be easily verified using the facts that $a_n=n^{1/\alpha}L(n)$  and $L(k)/L(n)\ll (k/n)^{\varepsilon}$ for any $\varepsilon>0$ where $L$ is a slowly varying function. Thus by Theorem \ref{theorem2} \eqref{eq5} is equivalent to \eqref{eq6} with $Y_n\de X_n$ satisfying the conditions of Theorem \ref{theorem1}. Now applying Lemma \ref{lemma1} and Lemma \ref{lemma2} we get
\[
\frac 1{\log N}\sum_{n=1}^N \frac1n
\mathbf {I}\biggl(\dfrac{\mu}{a_n}
\sum_{k=1}^{[nt]}\log\left(\dfrac{S_k}{\mu k}\right)\le x\biggr)\asto \pp\left(\int_{0}^{t}\frac{\cl_{\alpha}(s)}{s}ds\le x\right) \quad \text{as} \quad N \toi.
\]
\end{proof}

\end{document}